\documentclass[submission]{tempclass}
\usepackage{geometry}                
\usepackage{graphicx}
\usepackage{amssymb}
\usepackage[capitalise]{cleveref}		
\usepackage{color,soul}

\usepackage[latin1]{inputenc}

\usepackage{amsmath}
\usepackage{amscd}
\usepackage{mathrsfs}
\usepackage{amscd}

\usepackage{setspace}

\newtheorem{thm}{Theorem}
\newtheorem{prop}[thm]{Proposition}

\newtheorem{lem}[thm]{Lemma}

\newtheorem{defn}[thm]{Definition}

\newcommand{\LR}{\Leftrightarrow}
\newcommand{\RA}{\Rightarrow}
\newcommand{\LA}{\Leftarrow}

\newcommand{\rapf}{\n $\RA:$\ }
\newcommand{\lapf}{\n $\LA:$\ }
\newcommand{\n} {\noindent}

\newcommand{\Pf}{\mathcal{P}_{\hspace{-0.2mm}\mathrm{fin}}   }
\newcommand{\UU}{\mathbb{U}}

\newcommand{\ddiam}{\delta_{\mathrm{diam}}}
\newcommand{\dstein}{\delta_{\mathrm{Steiner}}}

\renewcommand{\hat}{\widehat}
\title{A Universal Separable Diversity}

\author{David Bryant\addressmark{1}\thanks{Email: \email{david.bryant@otago.ac.nz}} 
\and Andr\'{e} Nies\addressmark{2}\thanks{Email: \email{andre@cs.auckland.ac.nz}}
\and Paul Tupper\addressmark{3}\thanks{Email: \email{pft3@sfu.ca}}}
\address{\addressmark{1}Dept.\ of Mathematics and Statistics, University of Otago, New Zealand.\\
\addressmark{2}Dept.\ of Computer Science, University of Auckland, New Zealand. \\
\addressmark{3}Dept.\ of Mathematics, Simon Fraser University, Canada.}

\keywords{Diversities, Urysohn space, Kat\v{e}tov functions, universality, ultrahomogeneity}

\begin{document}
\maketitle

\begin{abstract}
The Urysohn space is a separable complete metric space with two fundamental properties: (a) \emph{universality}: every separable metric space can be isometrically embedded in it; (b) \emph{ultrahomogeneity}: every finite isometry between two finite subspaces  can be extended to an auto-isometry of the whole   space. The Urysohn space is uniquely determined up to isometry within separable metric spaces by these two properties. 
We introduce an  analogue of the Urysohn space for diversities, a recently developed variant  of the  concept of a metric space.
In a diversity any finite set of points is assigned a non-negative value, extending the notion of a metric which only applies to unordered   pairs of points. 
We construct the unique separable complete diversity that it is ultrahomogeneous and universal  with respect to separable diversities.
\end{abstract}

\section{Introduction}

  Urysohn~\cite{urysohn1927espace} in 1927 constructed a remarkable metric space which is now named after him. The Urysohn space is the unique (up to isometry) separable complete metric space with the following two properties: (a) \emph{universality}: all separable metric spaces can be isometrically embedded within it; (b) \emph{ultrahomogeneity}: any isometry between two finite subspaces of the Urysohn space can be extended to an auto-isometry of the whole space.  

The property of universality is straightforward to grasp, and   holds for several    separable complete metric spaces,  such as $C[0,1]$. The property of ultrahomogeneity is less  known. Recall that homogeneity of a metric space means that given any two points $x, y$ in the space, there is an automorphism (or self-isometry) of the space that maps $x$ to $y$. Likewise, a space is 2-homogeneous if for every pair of pairs $(x_1,x_2)$ and $(y_1,y_2)$ such that $d(x_1,x_2)=d(y_1,y_2)$, there is an automorphism of the space taking $x_1$ to $y_1$ and $x_2$ to $y_2$. For any $k \geq 1$, $k$-homogeneity is defined similarly. Ultrahomogeneity just extends this property to any pair of isometric finite subsets of the space. An example of a complete separable ultrahomogeneous space is the separable infinite-dimensional Hilbert space $\ell_2$; see Melleray \cite{Melleray2008a}.
Urysohn established that the Urysohn space is the unique (up to isometry) separable metric space satisfying both universality and ultrahomogeneity \cite{Melleray2008a}.

Here we construct the analogue of the Urysohn space for {\em diversities}, a generalization of the concept of metric spaces wherein all finite subsets, and not just pairs of points, are assigned a non-negative value.
A diversity is a pair $(X,\delta)$ where $X$ is a set and $\delta$ is a function from the finite subsets of $X$ to $\mathbb{R}$ satisfying  
\begin{quotation}
\noindent (D1) $\delta(A) \geq 0$, and $\delta(A) = 0$ if 
and only if 
$|A|\leq 1$. \\
(D2) If $B \neq \emptyset$ then $\delta(A\cup B) + \delta(B \cup C) \geq \delta(A \cup C)$
\end{quotation}
for all finite $A, B, C \subseteq X$.  
Diversities were   introduced in \cite{Bryant12}. They   form an extension of the  concept of a metric space. Indeed, every diversity has an {\em induced metric}, given by $d(a,b) = \delta(\{a,b\})$ for all $a,b \in X$. Note also that  $\delta$ is {\em monotonic}: $A \subseteq B$ implies $\delta(A) \leq \delta(B)$. Also $\delta$ is {\em subadditive on sets with nonempty intersection}: $\delta(A \cup B) \leq \delta(A) + \delta(B)$ when $A \cap B \neq \emptyset$ \cite[Prop. 2.1]{Bryant12}.
We say that a diversity $(X,\delta)$ is complete if its induced metric $(X,d)$ is complete \cite{Poelstra13},  and that a diversity is separable if its induced metric is separable.

Our main goal is to construct the diversity analog  $(\mathbb{U}, \delta_{\mathbb{U}})$ of the Urysohn metric space. It is determined uniquely by being universal for separable diversities, and  ultrahomogeneous in the sense that isometric finite subdiversities are  automorphic. 

 The construction follows the same approach as Kat\v{e}tov's construction of the Urysohn universal metric space \cite{katetov1986universal}. Starting with any diversity $(X,\delta)$,  we denote by $E(X)$  the set of   one-point extensions of $X$.  Since $E(X)$ turns out to not be separable under the natural metric,  we instead consider extensions with \emph{finite support}, which provides a separable diversity $E(X,\omega)$ in which $(X,\delta)$ is naturally embedded. Repeating  this procedure we obtain a nested sequence of separable diversities. The  analogue of the Urysohn metric space is constructed as  the completion of the direct limit  of all these diversities. Finally we show that this complete separable diversity has the diversity analogue of  Urysohn's extension property, and hence  is universal and ultrahomogeneous. 
 
Another aspect of metric space theory that has been succesfully carried over to diversities is tight span theory \cite{dress84}. The tight span of a metric space $(X,d)$ can be viewed as the set of all \emph{minimal} members of $E(X)$. The connections between the tight span of a metric space, the space of all one-point extensions $E(X)$, and the Urysohn space are explored in \cite{aksoy2016}. It is shown, for example, that the tight span of the Urysohn space is not separable. Similar themes have been developed for diversities. In \cite{Bryant12} tight span theory was studied  for diversities, where the tight span of a diversity $(X,\delta)$ consists of all minimal one-point extensions with an appropriate diversity, and therefore can be seen as a subset of $E(X)$.


{\bf Acknowledgement}.  DB was supported by a University of Otago Strategic Grant. AN was partially supported by the Marsden fund of New Zealand. Some of this work was carried out while AN visited the Institute for Mathematical Sciences, Singapore. PT was supported by an NSERC Discovery Grant, an NSERC Discovery Grant and a Tier 2 Canada Research Chair.

\section{Background and Preliminaries}  \label{s: background}
%
Recall from above that any diversity $(X,\delta)$ has an induced metric $(X,d)$ where $d(a,b)=\delta(\{a,b\})$ for all $a,b \in X$. Conversely, given any metric space $(X,d)$,  consider the  diversities that have $(X,d)$ as an induced metric. Lower and upper bounds on the possible diversities that have $(X,d)$ as the induced metric are provided by the \emph{diameter diversity} and the \emph{Steiner diversity}, which we now introduce.

For any metric space $(X,d)$, the corresponding  diameter diversity $(X,\ddiam)$ is defined by
\[
\ddiam (A) = \max_{a,b \in A} d(a,b)
\] 
for all finite $A \subseteq X$.

On the other hand, given a metric space $(X,d)$, consider the weighted complete graph $(X,E,w)$ where $X$ is the set of vertices,  $E$    is the set of all unordered pairs of vertices, and $w$ assigns weight $d(a,b)$ to the edge $(a,b)$. A tree $T$ with vertices in $X$ \emph{covers} a finite set $A \subseteq X$ if $A$ is a subset of the vertices of $T$.  The Steiner diversity $(X,\dstein)$ is defined by letting $\dstein(A)$ be  the infimum, over all trees that cover $A$, of the total weight of the tree. 

The diameter diversity and the Steiner diversity of a metric space $(X,d)$ are important in that for any other diversity $(X,\delta)$ that has $(X,d)$ as an induced metric space we have
\[
\ddiam(A) \leq \delta(A) \leq \dstein(A),
\]
for all finite $A \subseteq X$ \cite{Bryant14}. 

The two diversities $\ddiam$ and $\dstein$  are determined purely by their values on pairs of points. We will show in the last section of this paper that the diversity analogue of the Urysohn metric space is neither a diameter diversity nor a Steiner diversity of any metric space. In particular, it is neither the diameter diversity nor the the Steiner diversity of the Urysohn metric space, even though it has the Urysohn metric space as its induced metric space.

\section{Analogue of Kat\v{e}tov functions} \label{s:Katetov}
For a metric space $(X,d)$, a Kat\v{e}tov function  $f \colon X \to \mathbb R$ describes a potential one-point extension of $X$ by a point $z$: a metric $\hat{d}$ on $X \cup \{ z\}$ extending $d$ is given by defining $\hat{d}(x,z) = f(x)$ for each $x \in X$.  By~\cite{katetov1986universal} we have
\begin{equation} \label{eqn:charKatetov} f \in E(X)  \LR \forall  x \forall y \,  |f(x) - f(y)| \le d(x,y) \le f(x)+ f(y).  \end{equation} 
 $E(X)$ is the set of Kat\v{e}tov functions, which form a metric space with the sup distance $d_\infty(f,g) = \sup_x |f(x) -g(x)|$. Identifying $x\in X$ with the function $y \mapsto d(x,y)$  isometrically  embeds $X$ into $E(X)$.
  
Let $(X,\delta)$ be a diversity. We will define its   extension $E(X)$ by adapting  Kat\v{e}tov's approach \cite{katetov1986universal}. 

\begin{defn} A function
$f \colon \Pf(X) \rightarrow \mathbb{R}$ is called \emph{admissible} if for some point $z$, $(X\cup\{z\}, \widehat \delta)$ is a 
diversity, where 
\[
\widehat \delta(A)=\delta(A), \ \ \ \ \widehat \delta(A \cup \{z\}) = f(A)
\]
for all finite $A \subseteq X$. The point $z$ may be in $X$.
\end{defn}

As before, each admissible function on $(X,\delta)$ corresponds to a way of extending $(X,\delta)$ by one   point $z$. 
We let $E(X)$ be the set of all admissible functions on $(X,\delta)$. We provide the analogue of \eqref{eqn:charKatetov}.
\begin{lem} \label{lemma:admissible_criteria}
A function $f \colon \Pf(X) \rightarrow \mathbb{R}$ is in $E(X)$ if and only if $f$ satisfies the following:\\
(i) $f(\emptyset)=0$,\\
(ii) $f(A) \geq \delta(A)$, for all $A$,\\
(iii) $f(A \cup C)+ \delta(B \cup C) \geq f(A \cup B)$, for all $A$, $B$, and $C$ with $C \neq \emptyset$\\
(iv) $f(A) + f(B) \geq f( A \cup B)$.
\end{lem}
\begin{proof}
\rapf Suppose $f$ is admissible, so $\widehat \delta$ is a 
diversity on $X \cup \{z\}$, and $f(A)=\widehat \delta(A \cup \{z\})$ for all $A \in \Pf(X)$. Then $\widehat{\delta}(\{z\})=0$ implies property (i). Monotonicity of $\widehat \delta(A)$ implies $f(A) = \widehat \delta(A \cup \{z\}) \geq \widehat \delta(A) = \delta(A)$, which is property (ii). The triangle inequality (D2) for $\widehat \delta$ gives, for all $C \neq \emptyset$,
\[
f(A \cup C) + \delta(B \cup C) = \widehat \delta(A \cup C \cup \{z\}) + \widehat \delta(B \cup C) \geq \widehat \delta(A \cup B \cup \{z\}) = f(A \cup B),
\]
which is property (iii).
Finally, using the triangle inequality for $\widehat{\delta}$ again gives
\[
f(A) + f(B) = \widehat{\delta}(A  \cup \{z\}) + \widehat{\delta}(B \cup \{z\}) \geq \widehat{\delta}(A \cup B \cup \{z\}) = f(A \cup B),
\]
which is property (iv).

\lapf Suppose now that  $f$ satisfies the properties (i) through (iv). If $f(\{x\})=0$ for some $x \in X$, let $z=x$. Otherwise let $z \not \in X$. Define $\widehat \delta$ on $X\cup\{z\}$ by 
\[
\widehat \delta(A) = \delta(A), \ \ \ \widehat \delta(A \cup \{z\}) = f(A) 
\]
for all finite $A \subseteq X$.  Note that if $z \in A$, we have $\delta(A) \leq  f(A) \leq \delta(A) + f(\{z\}) = \delta(A)$ using properties (ii) and (iii), so that the definition is consistent.  We need to show that $\widehat \delta$ is a 
diversity function. Since $\delta$ is a diversity, $\widehat \delta(\{x\})=0$ for all $x \in X$, and Property (i) gives that $\widehat \delta(\{z\})=0$.  For monotonicity, we need to show $\widehat \delta(A \cup \{y\}) \geq \widehat \delta(A)$
 for four different cases. First, if $z$ is not equal to $y$ and not in $A$, then it follows from monotonicity of $\delta$. Secondly, if $z=y$ and $z \not \in A$ then 
\[
\widehat \delta( A \cup \{z\}) = f(A) \geq \delta(A) = \widehat \delta(A)
\]
by property (ii). Thirdly, if $z \neq y$ and $z \in A$ then
\[
\widehat \delta(A \cup \{y\}) = f( A \setminus\{z\} \cup y) \geq f(A \setminus \{z\}) - \delta(\{y\})
= \widehat \delta(A) - 0
\]
by property (iii). Fourthly, if $z =y$ and $ z\in A$ we have $A \cup\{y\}=A$ and hence $\widehat \delta(A \cup \{y\}) = \widehat \delta(A)$.
To show that $\widehat \delta$ is subadditive on intersecting sets we again have several cases.
If $z \not \in A$ and $z\not \in B$ then 
\[
\widehat \delta(A) + \widehat \delta(B) = \delta(A) + \delta(B) \geq \delta(A \cup B) = \widehat \delta(A \cup B).
\]
If $z$ is in $A$ but not in $B$ then
\[
\widehat \delta(A) + \widehat \delta(B) = f(A \setminus \{z\}) + \delta(B) \geq f( A \setminus \{z\} \cup B) = \widehat \delta(A \cup B).
\]
using property (iii).
Likewise if $z$ is in $B$ but not in $A$.
Finally, suppose $z \in A \cap B$. Then
\[
\widehat \delta(A) + \widehat \delta(B) = f( A \setminus \{z\}) + f(B \setminus \{z\}) \geq 
f( (A \cup B) \setminus \{z\} ) = \widehat \delta(A \cup B).
\]
Hence $\widehat \delta$ is subadditive on intersecting sets. Together with monotonicity this gives the triangle inequality for diversities.
\end{proof}

Analogous to the metric $d_\infty$ in  Kat\v{e}tov's construction, we define a diversity function $\widehat{\delta}$ on $E(X)$.  
The motivating idea for our choice of function is that since every admissible function $f$ corresponds to extending a diversity by an additional point $z$, considering admissible functions $f_1, \ldots, f_k$ should require us to extend the diversity by points $z_1, \ldots, z_k$ simultaneously, giving a new diversity $\delta_E$ defined on $X \cup \{z_1, \ldots, z_k\}$. 
This diversity must  coincide with $\delta$ on $X$, and also satisfy that   $f_i(A) =\delta_E(A \cup \{z_i\})$ for $i = 1, \ldots, k$. Once we have fixed a choice of $\delta_E$ given these constraints, we let 
\[
\widehat{\delta} ( \{f_1, \ldots, f_k \}) =  \delta_E( \{z_1, \ldots, z_k \}).
\]

One choice for $\widehat \delta$ that turns out to generalize from the metric case nicely is to let  $\widehat \delta$ to be the minimum diversity satisfying the constraints
\begin{equation} \label{eq:div_conditions}
\widehat \delta(A) = \delta(A), \ \ \ \ 
\widehat \delta(A\cup \{z_i\}) = f_i(A),\ \ i = 1, \ldots, k,
\end{equation}
for all finite $A \subseteq X$. We now describe how to obtain an explicit expression for $\widehat \delta$.

We say that a collection of finite subsets $E_1, \ldots, E_k$  is \emph{connected} if, when we partition $E_1, \ldots, E_k$ into two non-empty collections of sets,  there is an $E_{i}$ on one side of the partition and an $E_j$ on the other side of the partition such that $E_i \cap E_j \neq \emptyset$. Equivalently, define a graph with $v_1, \ldots, v_k$ corresponding to $E_1, \ldots, E_k$ and there is an edge between $v_i$ and $v_j$ if and only if $E_i \cap E_j \neq \emptyset$. Then the collection of sets is \emph{connected} if and only if the graph is connected.

To determine  $\widehat \delta$, we first obtain some lower bounds on $\widehat \delta(\{z_1, \ldots, z_k \})$. Choose any $j$ from $1, \ldots, k$. For $i \neq j$ choose finite subsets $A_i$ of $X$.
The sets $A_i \cup \{z_i\}$, $i \neq j$ together with $\{z_1,\ldots, z_k\}$ form a connected cover of the set $\{z_j\} \cup \bigcup_{i \neq j} A_i$.  So by the triangle inequality for diversities, we should have that 
\[
\widehat \delta \left(\{z_j\} \cup \bigcup_{i \neq j} A_i\right) \leq \widehat \delta \left(\{z_1,\ldots, z_k\}\right) + \sum_{i \neq j} 
\widehat \delta (A_i \cup \{z_i\}).
\]
Putting this into terms of admissible functions we get
\[
f_j \left( \bigcup_{i \neq j} A_i \right) \leq \widehat \delta( \{f_1, \ldots, f_k \}) + \sum_{i \neq j} f_i (A_i).
\]
This gives the following lower bound on $\widehat \delta$:
\[
\widehat \delta( \{f_1, \ldots, f_k \}) \geq f_j( \bigcup_{i \neq j} A_i ) - \sum_{i \neq j} f_i (A_i).
\]
Now this bound must hold for each choice of $j$ and $A_i$ for $i \neq j$, which suggests the following definition of $\widehat \delta$ on $E(X)$:
\begin{equation}
\widehat \delta(\{f_1, \ldots, f_k\}) = \max_{j = 1, \ldots, k} \sup_{A_1, \ldots, A_k} \left\{
f_j( \cup_{i \neq j} A_i) - \sum_{i \neq j} f_i(A_i) 
\right\} \label{eq:hatdelta}
\end{equation} 
where all $A_i$ are finite subsets of $X$. We define $\widehat \delta(\emptyset)$ and $\widehat \delta(\{f\})$ to be zero, for all $f \in E(X)$. 
Theorem~\ref{thm:proof_of_diversity} below shows that this is a diversity on $E(X)$ which extends $(X,\delta)$ naturally. The considerations above show that it is the minimal diversity satisfying conditions \eqref{eq:div_conditions}.
Note that if  $k=2$ we simply have 
\[
\widehat \delta(\{f_1,f_2\}) = \sup_{B \mbox{ finite}} |f_1(B) - f_2(B)|.
\]

We now make some comparisons between $(E(X),\widehat \delta)$ and the tightspan diversity of $(X,\delta)$ defined in \cite{Bryant12}.
Points in $E(X)$ correspond to one-point extensions of the diversity $(X,\delta)$; points in the tightspan $T(X)$ of $X$ correspond to \emph{minimal} one-point extensions of $(X,\delta)$. Thus $T(X) \subseteq E(X)$. By Lemma 2.6 of \cite{Bryant12}, the tightspan diversity $\delta_T$  equals the restriction of $\widehat \delta$ to $T(X)$, noting that on $T(X)$  the $k$ different suprema we are taking the maxima over in \eqref{eq:hatdelta} are all identical, and hence the expression simplifies.

\begin{thm} \label{thm:proof_of_diversity}
$(E(X), \widehat \delta)$ is a diversity, and  $(X,\delta)$ is embedded in $(E(X),\widehat \delta)$ via the map $x \rightarrow \kappa_x$ where $\kappa_x(A) = \delta(A \cup \{x\})$.
\end{thm}
\begin{proof}
First note that by construction we get $\widehat \delta(\emptyset) = 0$ and $\widehat \delta(\{f\})=0$ for any single admissible function $f$.
If $f$ and $g$ are distinct members of $E(X)$ then $|f(B) - g(B)| > 0$  for some finite $B$, so $\widehat \delta(\{f,g\})>0$. 

To show monotonicity of $\widehat \delta$, note that restricting the size of the set of elements of $E(X)$ restricts the number of functions that can take the first position in the supremand and restricts that the corresponding $A_i$ must be the empty set. So $\widehat \delta$ can only decrease when removing elements from a set.

To show that $\widehat \delta$ satisfies the triangle inequality, let $F$ and $G$ be two finite sets of functions in $E(X)$  and let $h$ be another admissible function. Let arbitrary $\epsilon>0$ be given. By the definition of $\widehat \delta$ there is a collection of sets $A_i$ and $B_i$ as well an index $j$ such that 
\[
\widehat \delta(F \cup G) - \epsilon  \leq  f_j\left( \cup_{i\neq j} A_i \bigcup \cup_k B_k\right) - 
\sum_{i \neq j} f_i(A_i) - \sum_k g_k(B_k).
\] 
We can assume, without loss of generality, that $f_j$ belongs to one of the admissible functions in $F$. Adding and subtracting $h(\cup_k B_k)$ gives
\begin{eqnarray*}
\widehat \delta(F \cup G) - \epsilon & \leq &
f_j( \cup_{i \neq j} A_i  \bigcup \cup_{k} B_k) - \sum_{i \neq j} f_i(A_i) - h( \cup_k B_k) \\ 
 &  &  + h(\cup_k B_k) - \sum_{k} g_k(B_k)
 \\ 
 & \leq & \widehat \delta( F \cup \{h\}) + \widehat \delta( G \cup \{h\}).
\end{eqnarray*}
This is true for all $\epsilon>0$ so the triangle inequality holds.

Finally, since $\widehat \delta$ restricted to $T(X)$ is the tight span diversity $\delta_T$, and, by Theorem 2.8 of \cite{Bryant12},  $\delta_T(\kappa(A)) = \delta(A)$ for all $A \in \Pf(X)$. We conclude that $\widehat \delta(\kappa(A)) = \delta(A)$ for all $A \in \Pf(X)$.

%
\end{proof}

\section{Extensions, supports, and  $E(X,\omega)$}

Recall that a diversity $(X,\delta)$ is separable if the underlying induced metric is separable. Analogous to the metric case, the diversity $(E(X),\widehat \delta)$ need not be separable, even when $(X,\delta)$ is. To get a separable but sufficiently rich subspace of $E(X)$ we develop the concepts of support for admissible functions of diversities.

\begin{defn}
Let $(X,\delta)$ be a diversity, let $S \subseteq X$, and let $f \in E(S)$. We define the \emph{extension of $f$ to $X$} as
\begin{equation}
f_S^X(A) = \inf \left\{  f(B) + \sum_{b\in B}  \delta(A_b \cup \{b\})  \colon \mbox{ finite }B \subseteq S, \bigcup_{b \in B} A_b = A \right\}. \label{eq:fSX}
\end{equation}
for finite $A \subseteq X$. 
\end{defn}

The definition of $f^X_S$ can be viewed as a  \emph{one-point amalgamation}. Amalgamation is a concept from algebra that also occurs in model theory. Two structures that share a common substructure are simultaneously embedded into a larger structure. Here the two structures are   diversities. One is $(X,\delta)$, and the other is the diversity on $S \cup \{z\}$ corresponding to  the function $f$, where $z$ is a single point that may or may not be in $S$. Since $S \subseteq X$, the two diversities overlap (have a common substructure) on $S$.  In Lemma~\ref{lemma:extend_from_support} below, we show that  $f^X_S$ is an admissible function, and hence it corresponds to a diversity on $X \cup \{z\}$ that extends both $(X,\delta)$ and the diversity on $S \cup \{z\}$ corresponding to $f$. Furthermore, it is the maximal such extension.
 
\begin{defn}
Let $g$ be an admissible function on $(X,\delta)$ and $S \subseteq X$ be nonempty.
If $g=f_S^X$ for some $f \in E(S)$ we say that $g$ \emph{has support} $S$.
We say that  $f$ is \emph{finitely supported} if it has some finite support $S$.
\end{defn}

In the following we use $g \upharpoonright S$ to denote the restriction of $g$ to $S$.

\begin{lem} \label{lemma:extend_from_support}
Let $(X,\delta)$ be a diversity, let $S \subseteq X$, and let $f \in E(S)$. Then $f_S^X$ is an admissible function on $X$ such that $f_S^X(A) = f(A)$ for all finite $A \subseteq S$. Furthermore, it is the unique maximal such extension, in that for any admissible function $g$ such that $g \upharpoonright S =f$, we have $g(A)\leq f_S^X (A)$ for any finite set $A \subseteq X$.
\end{lem}
\begin{proof}
We first show that $f_S^X$ is an admissible function on $X$ by checking conditions (i) through (iv). (i) follows
from the non-negativity of $f$ and $\delta$ and setting $B$ to be the empty set. 
To show (ii) we  use property (ii) for $f$ to see that the expression inside the infimum for $f_S^X(A)$ satisfies
 \[
 f(B) +  \sum_{b\in B} \delta( A_b \cup \{b\} )
 \geq  \delta(B) +  \sum_{i=1}^k \delta( A_b \cup \{b\} ) 
 \geq  \delta(\bigcup_{b\in B} A_b) = \delta(A),
\]
where we have used condition (ii) for $f$ and then the triangle inequality for diversities. 
For condition (iii), let $C$ be an arbitrary nonempty set.  Then
\begin{eqnarray*}
f_S^X(A \cup C) + \delta(B \cup C) 
&= & 
\inf_{D \subseteq S} \,  \inf_{\cup A_d = A \cup C } \left\{ f(D) + \sum_{d \in D} \delta( A_d \cup \{d\} ) \right\} + \delta( B \cup C).
\end{eqnarray*} 
For each such choice of $\{A_d\}, d \in D$,
let $e$ be an element of $D$ such that $A_e$ and $C$ intersect. Then from the triangle inequality $\delta(A_e \cup \{e\}) + \delta( B \cup C) \geq \delta(A_e \cup B \cup C \cup \{e\})$.  So
\[
f_S^X(A \cup C) + \delta(B \cup C) 
 \geq   
\inf_{D \subseteq S} \,  \inf_{\cup A_d = A \cup C } \left\{ f(D) + \sum_{d \neq e} \delta( A_d \cup \{d\} )+ \delta( A_e \cup  B \cup C \cup \{e\})  \right\} 
\] 
Now the union of the sets $A_d$ for $d \neq e$ together with $A_e \cup B \cup C$ is $A \cup B \cup C$. So from the definition of $f_S^X(A \cup B \cup C)$ we get 
\[
f_S^X(A \cup C) + \delta(B \cup C) \geq f_S^X(A \cup B \cup C) \geq f_S^X(A \cup B)
\]
the last step following from monotonicity of $f_S^X$.

For condition (iv) note that for all finite $A, B \subseteq X$
\begin{eqnarray*}
f_S^X(A \cup B) & = &  \inf_{D \subseteq S} \  \inf_{ \bigcup_{d\in D} G_d = A \cup B} \left\{ f(D) +  \sum_{d\in D} \delta( G_d \cup \{d\})  \right\}  \\
& \leq & \inf_{E, F \subseteq S} \ \ \inf_{\cup_{e \in E} A_e = A} \ \ \inf_{\cup_{f \in F} B_f = B}  \left\{ f(E \cup F) +  \sum_{e\in E } \delta( A_e \cup \{e\}) + \sum_{f \in F} \delta( B_f \cup \{f\})  \right\} 
\end{eqnarray*}
where we have used the fact that the infimum increases because we restricted it to the case when $D$ is a union of two sets, one of which indexes a cover of $A$ and the other indexes a cover of $B$ (and we've allowed some double counting of indices).
Now since $f(E \cup F) \leq f(E) + f(F)$, we can decompose the infimum to get 
$f_S^X( A \cup B)  \leq  f_S^X(A) + f_S^X(B)$, as required.

Next we show that $f_S^X$ is an extension of $f$ in that $f_S^X(A) = f(A)$ for all finite $A \subseteq S$. First note that taking $B=A$ and $A_b =\{b\}$ for all $b \in B$ in the definition of $f_S^X$ gives that $f_S^X(A) \leq f(A)$. Secondly, if we use condition (iii) of admissible functions repeatedly in the expression in the infimum we get $f_S^X(A) \geq f(A)$, giving the result.

To show that $f_S^X$ has $S$ as a support, just replace the $f$ with $f_S^X$ in the definition of $f_S^X$ and see that it does not change the result, which you can do since $f$ and $f_S^X$ agree on all subsets of $S$.
\end{proof}

Let $f$ be any admissible function on $(X,\delta)$ and let $S=X$.
Repeated use of property (iii) of admissible functions shows 
\[ f(B) + \sum_{b\in B}  \delta(A_b \cup \{b\})  \geq f(A)\]
in equation \eqref{eq:fSX}, so equality holds for all $A$. Hence all admissible functions on $(X,\delta)$ have $X$ as a support.

We define
\[
E(X,\omega)= \{ f \in E(X) \colon f \mbox{ is finitely supported} \}
\]
Note that  $E(X,\omega)$ is a subspace of $E(X)$, and that  $\kappa_x$ is finitely supported for each $x \in X$ since it has support $\{x\}$. So $E(X,\omega)$ with diversity $\widehat \delta$ is still an extension of the given  diversity $(X,\delta)$.

We now show that $(E(X,\omega),\widehat \delta)$ is separable whenever $(X,\delta)$ is.  Recall that separability of a diversity just means separability of the induced metric space.


\begin{lem} \label{lem:finite_gives_separable}
Let   $(X,\delta)$ be a diversity with $|X|=n < \infty$. Then $E(X)=E(X,\omega)$ is homeomorphic to a closed subspace of $\mathbb{R}^{\Pf(X)}$.
\end{lem}
\begin{proof}
Every function $f \in E(X)$ can be naturally identified as an element of $\mathbb{R}^{\Pf(X)}$. $E(X)$ corresponds to those elements of $\mathbb{R}^{\Pf(X)}$ with the element satisfying the conditions of an admissible function. Since these conditions consist of a linear equality and some non-strict linear inequalities, the subset of $E(X)$ is closed in $\mathbb{R}^{\Pf(X)}$. We just need to show that the metric induced by $\widehat \delta$ is homeomorphic to the Euclidean metric.

Since $
\widehat \delta(\{f,g\}) = \sup_{B \in \Pf(X)} |f(B) - g(B)|$, 
which is the $\ell_\infty$ norm, this  gives the same topology as the Euclidean norm in 
$\mathbb{R}^{\Pf(X)}$.
\end{proof}

\begin{lem} \label{lemma:admissible_continuity}
Let $f$ be an admissible function on the diversity $(X,\delta)$. Let $A= \{ a_1, \ldots, a_n\}$ and $B= \{ b_1, \ldots, b_n\}$ be subsets of $X$, where $\delta( \{a_i, b_i\} ) \leq \epsilon$ for $i =1, \ldots, n$. Then
\[
| f(A) - f(B) |  \leq n \epsilon.
\]
\end{lem}
\begin{proof}
Using property (iii) of admissible functions
\begin{eqnarray*}
f(A) & = & f( \{ a_1, \ldots, a_n \}) \\
      & \leq & f( \{ b_1, a_2, \ldots, a_n \}) + \delta(\{a_1,b_1\}) \\
      & \leq & \cdots \\
      & \leq & f( \{ b_1, \ldots, b_n \}) + \sum_{i=1}^n \delta(\{a_i,b_i\}) \\
      & = & f(B) + n \epsilon.
\end{eqnarray*}
Applying the same argument with $B$ and $A$ reversed gives $f(B) \leq f(A) + n \epsilon$.  \end{proof}

\begin{thm}
Let $(X,\delta)$ be a separable diversity. Then $(E(X,\omega),\widehat \delta)$ is a separable diversity.
\end{thm}
\begin{proof}
Let $D$ be a countable dense set in $(X,\delta)$. We will show that  $(E(D,\omega),\widehat \delta_{D})$ is separable, and that $(E(D,\omega),\widehat \delta_{D})$ can be densely embedded in $(E(X,\omega),\widehat \delta_{X})$.

To show that $(E(D,\omega),\widehat \delta_{D})$ is separable, note that it is the union, over all finite subsets $S \subseteq D$, of the extensions of $(D,\delta)$ with support $S$. Since each set of extensions is separable (being homeomorphic to a closed subset of a finite-dimensional Euclidean space by Lemma~\ref{lem:finite_gives_separable}), and there are only countably many of them, $(E(D,\omega),\widehat \delta_{D})$ is separable.  

To show that $(E(D,\omega),\widehat \delta_{D})$ is densely embeddable in $(E(X,\omega),\widehat \delta_{X})$,  we define the embedding $\gamma$. For $f \in E(D,\omega)$ we will define $\gamma f= \hat{f} \colon \Pf (X) \rightarrow \mathbb{R}$ via $\hat{f} = f_D^X$.
From Lemma \ref{lemma:extend_from_support} we have that $\hat{f}$ is an admissible function on $X$, $\hat{f}$ is an extension of $f$, and $D$ is a support of $\hat{f}$.

Next we need to show that for any finite set $F$ of admissible functions on $D$
\[
\widehat \delta_{X} ( \gamma F ) = \widehat \delta_{X} (F).
\]
First note that 
\begin{eqnarray*}
\widehat \delta_{X}(\gamma F) & = & \max_j \sup_{A_1, \ldots, A_k \subseteq X} \left\{ 
\gamma f_j( \cup_{i \neq j} A_i ) - \sum_{i \neq j} \gamma f_i(A_i) \right\} \\
& \geq & \max_j \sup_{A_1, \ldots, A_k \subseteq D} \left\{ 
\gamma f_j( \cup_{i \neq j} A_i ) - \sum_{i \neq j} \gamma f_i(A_i) \right\} \\
& = &  \max_j \sup_{A_1, \ldots, A_k \subseteq D} \left\{  f_j( \cup_{i \neq j} A_i ) - \sum_{i \neq j}  f_i(A_i) \right\} \\
& = & \widehat \delta_{D} (F),
\end{eqnarray*}
where we have used that $D$ is a subset of $X$ and that $\gamma f$ agrees with  $f$ on 
$D$. To show conversely that $\widehat \delta_{X}(\gamma F) \leq \widehat \delta_{D}(F)$, we need to show that for any choice of $j$ and finite $A_1, \ldots, A_k \subseteq X$, we can find finite  $B_1, \ldots, B_k$ so that $\gamma f_j( \cup_{i \neq j} B_i )$ is arbitrarily close to $\gamma f_j( \cup_{i \neq j} A_i)$ and $\gamma f_i(B_i)$ is arbitrarily close to $\gamma f_i(A_i)$ for all $i \neq j$. That such $B_i$ exist follows from the density of $D$ in $X$ and Lemma~\ref{lemma:admissible_continuity}.

We have shown that the map $\gamma \colon E(D,\omega) \rightarrow E(X,\omega)$ is an embedding. We still need to show that it is a dense embedding. Let $f \in E(X,\omega)$. Suppose $f$ has finite support $S$, with $|S|=n$ and elements $s_1, \ldots, s_n$. For any $\epsilon>0$, find a $T \subseteq D$ with $|T|=n$  elements $t_1, \ldots, t_n$ such that for any subindices $i_1,\ldots, i_m$ of $1,\ldots, n$ we have
\[
|f(\{t_{i_1}, \ldots, t_{i_m}\}) - f(\{s_{i_1}, \ldots, s_{i_m} \}) | < \epsilon.
\]
(This is possible by Lemma~\ref{lemma:admissible_continuity}.)
Now $f$ restricted to $T$ is still an admissible function. We want to extend it to all of $D$. For any finite subset $A$ of $D$, define $g = (f \upharpoonright T)_T^D$.
By Lemma~\ref{lemma:extend_from_support}, $g$ is an admissible function on $(D,\delta)$, it is an extension of $f \upharpoonright T$, and it has support $T$. Now we let $\hat{g}=\gamma g$ be the image of $g$ under the embedding. We need to show that $\hat{g}$ is close to $f$.

The functions $\hat{g},f \colon \Pf(X) \rightarrow \mathbb{R}$ agree on subsets of $T$, but $\hat{g}$ is supported on $T$ and $f$ is supported on $S$. Let $A$ be an arbitrary finite subset of $X$.
Since $T$ is finite, we have for some $B \subseteq T$, $B=\{t_{i_1},\ldots, t_{i_m}\}$  and $\{A\}_{b \in B}$ with $\cup_{b \in B} A_b = A$
\begin{eqnarray*}
\hat{g}(A) &  \geq & f(B) + \sum_{b \in B} \delta( A_b \cup \{b\}) - \epsilon \\
 & \geq  & f(C)-\epsilon + \sum_{c \in C} \delta(A_c \cup \{c\})  -\epsilon \\
 & \geq & f(A) - 2 \epsilon
\end{eqnarray*}
where $C \subset S$ and $C = \{ s_{i_1}, \ldots, s_{i_m} \}$.  A similar argument starting with $f(A)$ gives $f(A) \leq \hat{g}(A) - 2 \epsilon$. Together we have $|\hat{g}(A) - f(A) | \leq 2\epsilon$ for all finite $A \subseteq X$ and so $\widehat \delta_X ( \{ \hat{g} , f\}) \leq 2 \epsilon$ can be made arbitrarily small
as required.
\end{proof}

\section{Construction of the diversity analogue of the Urysohn metric space.}

Here we define the diversity analogue of the Urysohn metric space. We show that it is the unique universal Polish diversity. We also show that it is ultrahomogeneous. 

In what follows we will need the following lemma.
For each $k\ge 1$, let $\delta_k$ be  the function  that sends $(a_1, \ldots, a_k)$ to $\delta(\{ a_1, \ldots, a_k\})$. 

\begin{prop} \label{prop:1-Lipschitz} Let $(X, \delta)$ be a diversity. For each $k$, the function $\delta_k$ is 1-Lipschitz in each argument. \end{prop}
\begin{proof}
Consider varying the $i$th argument of $\delta_k$ from $x_i$ to $x_i'$. We know from the triangle inequality that 
\begin{eqnarray*}
\delta_k(x_1, \ldots, x_i, \ldots, x_k) & = & \delta(\{x_1, \ldots, x_i, \ldots, x_k\}) \\
 & \leq & \delta(\{x_1, \ldots, x_i', \ldots, x_k\}) + \delta(\{ x_i, x_i'\}) \\
 & = & \delta_k(x_1, \ldots, x_i', \ldots, x_k) + d(x_i,x_i').
\end{eqnarray*}
Similarly, $\delta_k(x_1, \ldots, x_i', \ldots, x_k) \leq \delta_k(x_1, \ldots, x_i, \ldots, x_k) +
d(x_i,x_i')$. So 
\[
|  \delta_k(x_1, \ldots, x_i, \ldots, x_k) - \delta_k(x_1, \ldots, x_i', \ldots, x_k)| \leq d(x_i, x_i')
\]
as required.
\end{proof}

\begin{defn}
A diversity $(X,\delta)$ has the extension property if for any finite subset $F$ of $X$ and any admissible function $f$ on $F$, there is $x \in X$ such that $f(A) = \delta(A \cup\{x\})$ for any finite $A \subseteq F$.
\end{defn}

The extension property on metric spaces is also known as the Urysohn property \cite{Gao2009}.

\begin{defn}
We say a diversity is Polish if its induced metric space is Polish, i.e. it is separable and complete.
\end{defn}

\begin{lem} \label{lem:cont_extense}
Let $(X,\delta_X)$ and $(Y,\delta_Y)$ be diversities where $X$ is separable with a dense subset $D_X$ and $Y$ is complete. Any isomorphism from $D_X$ into $Y$ can be extended to an isomorphism from $X$ into $Y$. 
\end{lem}
\begin{proof}
Let $\phi$ be an isomorphism from $D_X$ into $Y$.
Since $\phi$ preserves the diversity  it also preserves the induced metrics between the two sets and is hence a uniformly continuous map. This means we can extend it to a uniformly continuous function $\bar{\phi}$ between $X$ and $Y$. To show that $\bar{\phi}$ is an isomorphism, let $A \subset X$ be an arbitrary finite set, with $A= \{ a_1, \ldots, a_n \}$. For each $k=1,\ldots,n$, let $a^1_k,a^2_k,a^3_k,\ldots$ be a sequence in $D_X$ such that with $a^i_k \rightarrow a_k$ as $i \rightarrow \infty$. We define $A^i = \{a^i_1,a^i_2,\ldots,a^i_n\}$.
\begin{eqnarray*}
\delta_Y( \bar \phi (A) ) &= & \delta_Y ( \bar \phi ( \lim_i A^i )) \\
&= & \delta_Y ( \lim_i \bar \phi (A^i)) \\
& = & \lim_i \delta_Y (\bar \phi(A_i)) \\
& =& \lim_i \delta_X ( A_i) \\
& = & \delta_X ( \lim_i A_i)  = \delta_X (A).
\end{eqnarray*}
where we have used the uniform continuity of $\delta_X$ and $\delta_Y$, by Proposition~\ref{prop:1-Lipschitz}.
\end{proof}

\begin{thm} \label{thm:urysohn_unique}
Let $(X,\delta_X)$ and $(Y,\delta_Y)$ be Polish diversities with the extension property. Then $(X,\delta_X)$ and $(Y,\delta_Y)$ are isomorphic.
\end{thm}
\begin{proof}
We mostly follow the proof of \cite[Thm.\ 1.2.5]{Gao2009}.
Let $\{ x_0, x_1, x_2, \ldots\}$ be a dense set in $X$ and let $\{y_0, y_1, y_2, \ldots\}$ be a dense set in $Y$. We will define a diversity isomorphism between these dense sets and then extend it to the whole space.

We will construct a sequence of partial diversity isomorphisms $\phi_0, \phi_1, \phi_2, \ldots$
Let $\phi_0$ be defined on the single point $x_0$ so that $\phi_0(x_0)=y_0$.

 At stage $n>0$, suppose that $\phi_{n-1}$ has been defined so that $\{x_0, \ldots, x_{n-1}\} \subseteq \mbox{dom}(\phi_{n-1})$ and $\{y_0, \ldots, y_{n-1} \} \subseteq \mbox{range}(\phi_{n-1})$. If $x_n \in \mbox{dom}(\phi_{n-1})$ then we let $\phi'=\phi_{n-1}$. Otherwise, let $F = \mbox{range}(\phi_{n-1})$ and consider the admissible function on $F$ defined by
\[
f(A) = \delta_X( \phi_{n-1}^{-1}(A) \cup \{ x_n\})
\]
for finite $A \subseteq F$.
By the extension property  of $Y$ there is $y \in Y$ so that 
\[
\delta_Y(A \cup \{y\}) = f(A) = \delta_X(\phi_{n-1}^{-1} (A)  \cup \{ x_n\}).
\]
We extend $\phi_{n-1}$ to $\phi'$ by defining $\phi'(x_n)=y$. Now if $y_n \in \mbox{range}(\phi')$ then we let $\phi_n = \phi'$ and go on to the next stage.  Otherwise apply the above argument to $\phi'^{-1}$ and use the extension property of $X$ to obtain an extension of $\phi'$. Define $\phi_n$ to be this extension. We have thus finished the definition of $\phi_n$. Let $\phi$ be the union of all $\phi_n$ we have defined. Then it has the required properties.

Finally we use Lemma~\ref{lem:cont_extense} to extend $\phi$ between $\{x_0, x_1, x_2, \ldots\}$ and $\{y_0, y_1, y_2, \ldots \}$ to a isomorphism between $X$ and $Y$.
\end{proof}

The completion of a diversity is defined in \cite{Poelstra13}: we take the completion of the diversity's induced metric space, and then extend the original diversity function to this larger set using continuity.

Following \cite{Melleray2008a} we define the following weakened version of the extension property.
\begin{defn}
A diversity $(X,\delta)$  has the approximate extension property if for any finite subset $F$ of $X$, any admissible function $f$ on $F$, and any $\epsilon >0$, there is  an $x \in X$
such that  $|\delta(A \cup \{x\})-f(A)| \leq \epsilon$ for any $A \subseteq F$.
\end{defn}

\begin{lem} \label{lem:approx_extens}
If a separable diversity has the approximate extension property, then its completion has the approximate extension property.
\end{lem}
\begin{proof}
Let $(X,\delta)$ be a diversity that is the completion of dense subset $D$, where $(D,\delta)$ has the approximate extension property. Let $F$ be a finite subset of $X$, $f \in E(F)$, and $\epsilon>0$ be given. We need to find a point $y \in X$ such that $|\delta(A \cup \{y\}) - f(A) | \leq \epsilon$ for all $A \subseteq F$.

Order all non-empty subsets of $F$, $A_1, \ldots, A_{2^n-1}$ 
so that if $A_j \subseteq A_i$ then $j \geq i$. 
Let $\epsilon_0 = \epsilon/2(2^n +n)$.
Define a bijective map $\gamma$ from $F$ to $\gamma F \subseteq D$ so that for all nonempty $A \subseteq F$,  $|\delta(\gamma A) - \delta(A)|<\epsilon_0$, which is possible by Proposition~\ref{prop:1-Lipschitz}.

Define $g \colon \Pf(\gamma F) \rightarrow \mathbb{R}$ by  $g(\emptyset)=0$ and 
\[
g(\gamma A_i) = f(A_i) + \epsilon_{A_i}, \ \ \ \ \  \mbox{for } i = 1, \ldots, 2^n -1,
\]
where $\epsilon_{A_i} = i \epsilon_0$.  Note that $g$ is monotonic by construction.
We claim that $g \in E(\gamma F)$. 

To show $g \in E(\gamma F)$ we need to verify the four conditions of Lemma~\ref{lemma:admissible_criteria}. Condition (i) ($g(\emptyset)=0$) follows by definition.
To obtain condition (ii), note that for non-empty $A$, $g(\gamma A) = f(A) + \epsilon_A \geq \delta(A) + \epsilon_A \geq \delta(\gamma A) - \epsilon_0 + \epsilon_A \geq \delta(\gamma A)$. 
For condition (iii), we first observe that for any admissible function $f$ on $F$ and $C \neq \emptyset$ we have from the triangle inequality 
\[
f(A \cup C) + \delta(B \cup C) = f((A\cup C) \cup C) + \delta(B \cup C) \geq f(A \cup B \cup C).
\]
So,  given $A, B, C \subseteq F$,  with $C \neq \emptyset$,
\begin{eqnarray*}
g(\gamma A \cup \gamma C) + \delta(\gamma B \cup \gamma C) &\geq &  f(A \cup C) + \epsilon_{A \cup C} + \delta( B \cup C) - \epsilon_0 \\
& \geq  & f(A \cup B \cup C) + \epsilon_{A \cup C} - \epsilon_0 \\
& = &  g(\gamma A \cup \gamma B \cup \gamma C) - \epsilon_{A \cup B \cup C} + \epsilon_{A \cup C} - \epsilon_0 \\
& \geq & g( \gamma A \cup \gamma B)
\end{eqnarray*}
where we use the fact that $g$ is monotonic and that $A \cup B \cup C$ is later than $A \cup C$ on the list of subsets, and so $\epsilon_{A \cup B \cup C} + \epsilon_0 \leq \epsilon_{A \cup C}$. 
Now for condition (iv) we have
\begin{eqnarray*}
g(\gamma A) + g(\gamma B) & \geq & f(A) + \epsilon_A + f(B) + \epsilon_B \\
& \geq & f(A \cup B) + \epsilon_A + \epsilon_B \\
& \geq & g(\gamma A \cup \gamma B) - \epsilon_{A \cup B} + \epsilon_A + \epsilon_B\\
& \geq & g(\gamma A \cup \gamma B)
\end{eqnarray*}
since $\epsilon_{A \cup B} \leq \epsilon_A$.

So $g$ is admissible on $\gamma F$.
By the approximate extension property of $(D,\delta)$, there is a point $y$ such that $|g(\gamma A) - \delta(\gamma A \cup \{y\})| \leq \epsilon/2$ for all $A \subseteq F$. 

Now for any $A \subseteq F$ 
\begin{eqnarray*}
| f(A) - \delta( A \cup \{ y\}) | &  \leq  & | f(A) - g(\gamma A) | + | g(\gamma A) -\delta(\gamma A \cup \{y\})| + | \delta(\gamma A \cup \{y\}) - \delta( A \cup \{y\}) |  \\
& \leq & \epsilon_A + \epsilon/2 + n \epsilon_0  \leq 2^n \epsilon_0 + \epsilon/2 + n \epsilon_0 \leq \epsilon
\end{eqnarray*}
as required.
\end{proof}

\begin{lem} \label{lem:complete_extense}
Any complete diversity with the approximate extension property has the extension property.
\end{lem}
\begin{proof}
Our proof follows that of the metric case in Theorem 3.4 of \cite{Melleray2008a} and Theorem 1.2.7 of \cite{Gao2009}.

Let $(X,\delta)$ be a complete diversity with the approximate extension property. Let finite $F \subseteq X$ be given, and let $f \in E(F)$. It suffices to show there is a sequence $z_0, z_1, \ldots$ in $X$ such that for all $p$, $| \delta( A \cup \{z_p\}) - f(A) | \leq 2^{-p}$ for all $A \subseteq F$ and $\delta(\{z_p, z_{p+1}\} ) \leq 2^{1-p}$. Since $X$ is complete and $f$ is continuous, the sequence will have a limit $z \in X$ such that $\delta( A \cup \{z\}) = f(A)$ for all $A \subseteq F$.

By the approximate extension property of $(X,\delta)$ we can define $z_0$. To use induction, suppose we have $z_0, z_1, \ldots, z_p$ satisfying the conditions and we need to specify $z_{p+1}$. Let $f_p \in E(F)$ be defined by $f_p(A) = \delta( A \cup \{z_p\})$ for $A \subseteq F$. Note that for all $A$
\[
|f_p(A) - f(A) | = | \delta( A \cup \{z_p\}) - f(A) | \leq 2^{-p}.
\]
So $\widehat \delta( \{ f_p, f\}) \leq 2^{-p}$.

Now let $g_p$ be defined on $F \cup \{z_p\}$ by $g_p(A) = f(A)$, $g_p (A \cup \{z_p\}) =\widehat \delta( A \cup \{f_p,f\})$. This is in an admissible function on $F \cup \{z_p\}$ because it is realized by the points $F \cup \{f_p, f\}$ in $E(F)$. So by the approximate extension property there is a $z \in X$ that realizes $g_p$ with error at most $2^{-(p+1)}$. In other words
\[
|\delta( A \cup \{z\}) - g_p(A) | \leq 2^{-(p+1)}, \ \ \ \ \ |\delta(A \cup \{z,z_p\}) -g_p(A \cup \{z_p\}) | \leq 2^{-(p+1)}.
\]
The first inequality shows that $|\delta(A \cup \{z\}) - f(A) | \leq 2^{-(p+1)}$ and the second inequality shows that, choosing $A = \emptyset$
\[
\delta(\{z_p, z\}) \leq g_p(\{z_p\})  + 2^{-(p+1)} = \widehat \delta( \{f_p,f\}) + 2^{-(p+1)} \leq 2^{-p} + 2^{-(p+1)} \leq 2^{-p+1}.
\]
Now let $z_{p+1}=z$.
\end{proof}

\begin{thm} \label{thm:completion_still_has_extension_prop}
If $(X,\delta)$ is a separable diversity with the extension property then its completion also has the extension property.
\end{thm}
\begin{proof}
Since $(X,\delta)$ has the extension property, it certainly has the approximate extension property. By Lemma~\ref{lem:approx_extens} the completion of $(X,\delta)$ has the approximate extension property. Then by Lemma~\ref{lem:complete_extense} the completion of $(X,\delta)$ has the extension property, being complete.
\end{proof}

We now work towards defining a complete separable diversity with the extension property. We start with a given diversity $(X,\delta)$. We let $X_0=X$, $\delta_0= \delta$. Now, for $n>0$ we inductively define $(X_n, \delta_n)$ by letting $X_n = E(X_{n-1},\omega)$ with the diversity $\delta_n = \widehat \delta_{n-1}$. We define $(X_\omega,\delta_\omega)$ to be the union of all these diversities, which is well-defined because each $(X_n, \delta_n)$ is embedded in  $(X_{n+1},\delta_{n+1})$.

\begin{thm}
For any diversity $(X,\delta)$ the diversity $(X_\omega, \delta_\omega)$ has the extension property.
\end{thm}
\begin{proof}
Let $F$ be a finite subset of $X_\omega$, and let $f$ be an admissible function on $F$.  $F$ must be contained in $X_n$ for some $n$. By construction, there is some $x \in X_{n+1}$ such that $f(A) = \delta(A \cup \{x\})$ for all $A \subseteq F$. So there is such an $x$ in $X_\omega$.
\end{proof}

We define the diversity $(\mathbb{U},\delta_{\mathbb{U}})$ to be the completion of $(X_\omega,\delta_\omega)$ when $(X,\delta)$ is the trivial diversity of a single point.  By Theorem \ref{thm:completion_still_has_extension_prop}, $(\mathbb{U},\delta_{\mathbb{U}})$ also has the extension property.

We say that a Polish diversity is universal if any separable diversity is isomorphic to a subset it.

\begin{thm}
$(\mathbb{U}, \delta_{\mathbb{U}})$ is a universal Polish diversity.
\end{thm}
\begin{proof}
Let $(X,\delta)$ be an arbitrary separable diversity. We construct a sequence of partial isomorphisms whose union is the desired isomorphism. Let $x_0, x_1, x_2, \ldots$ be a dense sequence in $X$. Let $y$ be an arbitrary point in $\mathbb{U}$.   Let $\phi_0$ be defined on $\{ x_0 \}$ by $\phi_0(x_0)=y$. Now suppose that we have an isomorphism $\phi_n$ from $\{x_0, x_1, \ldots, x_n\}$ into $\mathbb{U}$, with $\phi(x_i)=y_i$ for $i=1,\ldots,n$. Define the admissible function on $\{y_0, \ldots, y_n\}$ for finite subset $A$ by $f(A)= \delta(\phi_n^{-1} (A) \cup x_{n+1})$. By the extension property, there is a point $y_{n+1}$ in $\mathbb{U}$ such that $\delta(  \phi_n^{-1} (A) \cup x_{n+1} ) =f(A)= \delta_{\mathbb{U}}( A \cup y_{n+1})$. Define $\phi_{n+1}$ by extending $\phi_n$ with one point with $\phi_{n+1}(x_{n+1})= y_{n+1}$. Now take the union of all of the $\phi_n$ to obtain an isomorphism between $\{x_0,x_1,x_2, \ldots\}$ and a subset of $\mathbb{U}$. By Lemma~\ref{lem:cont_extense} this isomorphism can be extended to all of $X$.
\end{proof}
 
A Polish diversity $(X,\delta)$ is ultrahomogeneous if  given any two isomorphic finite subsets $A, A' \subseteq X$, and any isomorphism $\phi \colon A \rightarrow A'$, there is an isomorphism of $(X,\delta)$ to itself that extends $\phi$.

\begin{thm}
$(\mathbb{U}, \delta_{\mathbb{U}})$ is  ultrahomogeneous. 
\end{thm}
\begin{proof}
This proof follows the same plan as Theorem~\ref{thm:urysohn_unique}. Let $A, A'$ be two isomorphic subsets of $\mathbb{U}$, with isomorphism $\phi$ between them. Let $\{ x_1, x_2, \ldots \}$ be a dense subset of $\mathbb{U} \setminus A$ and let $\{y_1, y_2, \ldots \}$ be a dense subset of $\mathbb{U} \setminus A'$. Let $\phi_0= \phi$. Suppose we have defined $\phi_{n-1}$ so that it is an isomorphism and $A \cup \{x_1, \ldots, x_{n-1} \} \subseteq \mbox{dom}(\phi_{n-1})$ and $A' \cup \{y_1, \ldots, y_{n-1}\} \subseteq \mbox{range}(\phi_{n-1})$. Following the proof of Theorem~\ref{thm:urysohn_unique} yields a suitable $\phi_n$. Taking the union of these $\phi_{n}$ and applying Lemma~\ref{lem:cont_extense} gives the desired isomorphism from $\mathbb{U}$ to itself that is an extension of $\phi$.
\end{proof}

\begin{thm}
Any ultrahomogeneous, universal Polish diversity has the extension property, and is thus isomorphic to $(\mathbb{U},\delta_{\mathbb{U}})$.
\end{thm}
\begin{proof}
Let $(X,\delta)$ be an ultrahomogeneous, universal Polish diversity. Let $F$ be a finite subset of $X$ and let $f$ be an admissible function on $F$. So we can define a diversity on $F \cup \{z\}$ for some $z$ such that $f(A) = \delta( A \cup \{z\})$ for $A \subseteq F$. Since $(X,\delta)$ is universal, there is an embedding $\phi$ taking $F \cup \{z\}$ into $X$. Let $F' = \phi(F)$.  Since $\phi$ is an isomorphism from $F$ to $F'$, there is an isomorphism $\phi'$ of  the whole space that extends $\phi$. Consider the point $\phi'^{-1} (z)$. It satisfies the property that $\delta( A \cup \phi'^{-1}(z) ) = f(A)$ for all $A \subseteq F$.
\end{proof}

\section{The relationship between $(\mathbb{U}, \delta_{\mathbb{U}})$ and the Urysohn metric space} 
We  denote the Urysohn metric space by $(\UU_{\mathtt m}, d)$. 
The induced metric space of the diversity $(\mathbb{U}, \delta_{\mathbb{U}})$ in, in fact, isometric to $(\UU_{\mathtt m}, d)$. 
We will show that  $(\mathbb{U}, \delta_{\mathbb{U}})$ is neither a diameter diversity nor a Steiner diversity. These notions were  recalled at  the beginning of Section~\ref{s: background}.


\begin{prop}
1.\ The   metric space induced by   $(\mathbb{U}, \delta_{\mathbb{U}})$ is isometric to $(\UU_{\mathtt m},d)$.\\
2.\ $(\mathbb{U}, \delta_{\mathbb{U}})$ is not a diameter diversity. \\
3.\  $(\mathbb{U}, \delta_{\mathbb{U}})$ is not a Steiner diversity.
\end{prop}
\begin{proof}
1. Recall that $(\UU_{\mathtt m},d)$ is up to isometry the unique separable complete metric space with the metric extension property. Since $(\mathbb{U}, \delta_{\mathbb{U}})$ is a separable complete diversity, its induced metric space is also separable and complete. 

It remains to show that the induced metric space has the metric extension property, which for $\mathbb U$ states that for any finite $A\subseteq \mathbb U$ and any $f \colon \mathbb U \rightarrow \mathbb{R}$ satisfying \eqref{eqn:charKatetov} at the beginning of Section~\ref{s:Katetov}
there is a $z \in \mathbb U$ such that 
$d(z,a)=f(a)$ for all $a \in A$.  Such a function~$f$ corresponds to a metric space $(A \cup \{z\},\hat{d})$, in that $f(a)=\hat{d}(a,z)$ for all $a \in A$, where $\hat{d}$ restricted to $A$ coincides with  the induced metric restricted to $A$. Define a diversity $\widehat{\delta}$ on $A \cup \{z\}$ by letting $\widehat{\delta}(B) = \delta_{\mathbb{U}}(B)$ and
 \[
 \widehat{\delta}(B \cup \{z\}) = \delta_{\mathbb{U}}(B) + \min_{b \in B} \hat{d}(b,z)
 \]
 for $B \subseteq A$.  Then $(A \cup \{z\}, \widehat{\delta})$ is a one-point extension of $(A, \delta_{\mathbb{U}})$. Since $(\mathbb{U},\delta_{\mathbb{U}})$ has the (diversity) extension property, $z$ can be identified with a point in $\mathbb{U}$. For any $a \in A$, 
\[
d(a,z)=\delta_{\mathbb{U}}(\{a ,z\}) = \widehat{\delta}(\{a, z\}) = \hat{d}(a,z)= f(a)
\]
as required. So the  metric induced on $\mathbb U$ has the extension property, and therefore is isometric to the Urysohn metric space.
 
 2. Consider the diversity on three points given by $X=\{a,b,c\}$, $\delta(a,b) = \delta(a,c) = \delta(b,c)=1$ and $\delta(a,b,c)=2$.  $(X,\delta)$ is not a diameter diversity, since in that case we would have $\delta(a,b,c)=1$. Since $(\mathbb{U},\delta_{\mathbb{U}})$ is universal, and $(X,\delta)$ is separable and complete, there is a subset of $\mathbb{U}$ that is isometric to $(X,\delta)$. Subsets of diameter diversities are still diameter diversities, so $(\mathbb{U},\delta_{\mathbb{U}})$ is not a diameter diversity. 
 
 3. Consider the diversity on three points given by $X=\{a,b,c\}$, $\delta(a,b)=\delta(a,c)=\delta(b,c)=1$ and $\delta(a,b,c)=1$. Since $(\mathbb{U},\delta_{\mathbb{U}})$ is universal and $(X,\delta)$ is  separable and complete, we can identify $(X,\delta)$ with a subset of $\mathbb{U}$. Suppose that $(\mathbb{U},\delta_{\mathbb{U}})$ is a Steiner diversity. Then we can find trees in $\mathbb{U}$ that cover $(a,b,c)$ and have total weight arbitrarily close to $\delta(a,b,c)=1$. Suppose we have a tree with total weight less than 1.25, covering $\{a,b,c\}$. We can assume that the tree has leaves $a, b, c$ and a single internal node $z$, which is possibly not distinct from $a, b, c$. Let the branches of the tree have lengths $\alpha, \beta, \gamma$, corresponding to the leaves $a,b,c$ respectively. Now we have 
 \[
 \alpha + \beta \geq 1, \  \  \beta + \gamma \geq 1, \  \ \alpha + \gamma \geq 1, \ \ \alpha + \beta + \gamma < 1.25.
 \]
 Summing the first three inequalities and dividing by 2 gives $\alpha + \beta + \gamma \geq 3/2$ which contradicts the final inequality. Therefore, $(\mathbb{U}, \delta_{\mathbb{U}})$ is not a Steiner diversity.
\end{proof}

\section{Questions}

Many questions that have been considered for the Urysohn metric space also make sense for the Urysohn diversity. 
For instance,  it is easy to 
 show that $(\mathbb{U}, \delta_{\mathbb{U}})$ is compact homogeneous, namely,  any isomorphic compact subdiversities  are automorphic. This  follows the construction in  Melleray~\cite[Section 4.5]{Melleray2008a} for the metric case.	 
It would be worthwhile to  
study the isometry   group of $(\mathbb{U}, \delta_{\mathbb{U}})$ along the lines of the results surveyed  in~\cite[Section 4.5]{Melleray2008a}. 

Another interesting avenue to explore is determining universal and ultrahomogeneous structures for restricted classes of diversities, as is described for metric spaces in \cite{van2010}.
For example, for $S \subseteq [0,\infty)$ let $\mathcal{M}_S$  be the class of all finite metric spaces with distances taking values in $S$. One can ask for each $S$ whether there is a metric space that is universal and ultrahomogeneous with respect to metrics in $\mathcal{M}_S$. For example, if $S=[0,\infty)$ then the corresponding structure is the Urysohn space. Delhomm\'e, Laflamme, Pouzet, and Sauer \cite{delhomme2007} give a complete characterization of which sets $S$ admit such a structure. These questions have natural analogues for diversities.

\bibliographystyle{plain}

\end{document}